\documentclass[11pt,leqno]{amsart}  

\usepackage[utf8]{inputenc}
\usepackage[T1]{fontenc}
\usepackage[english]{babel}
\usepackage[centertags]{amsmath}
\usepackage{amstext,amssymb,amsopn,amsthm}
\usepackage{nicefrac, esint}
\usepackage{mathrsfs}
\usepackage{dsfont}
\usepackage{bbm}
\usepackage{thmtools}

\theoremstyle{plain}
\newtheorem{thm}{Theorem}

\newtheorem{lem}[thm]{Lemma}
\newtheorem{prop}[thm]{Proposition}
\newtheorem{cor}[thm]{Corollary}

\theoremstyle{definition}
\newtheorem{defn}[thm]{Definition}
\newtheorem{example}[thm]{Example}
\newtheorem*{example*}{Example}

\newtheorem*{rem*}{Remark}
\newtheorem{rem}[thm]{Remark}

\newcommand{\N}{{\mathbb{N}}}

\newcommand{\R}{\mathbb{R}}

\newcommand{\Z}{\mathbb{Z}}

\newcommand{\DD}{D}
\newcommand{\wh}{\mathcal{W}}
\newcommand{\neio}{{\mathcal{N}_\Omega}}
\newcommand{\ind}{{\mathbf{1}}}

\newcommand{\Omint}{\Omega^{\rm int}}
\newcommand{\Omext}{\Omega^{\rm ext}}

\DeclareMathOperator{\interior}{int}
\DeclareMathOperator{\dist}{dist}
\DeclareMathOperator{\diam}{diam}

\DeclareMathOperator{\ext}{ext}
\DeclareMathOperator{\inr}{inr}

\numberwithin{equation}{section}

\def\kint_#1{\mathchoice%
          {\mathop{\kern 0.2em\vrule width 0.6em height 0.69678ex depth -0.58065ex
                  \kern -0.8em \intop}\nolimits_{\kern -0.4em#1}}%
          {\mathop{\kern 0.1em\vrule width 0.5em height 0.69678ex depth -0.60387ex
                  \kern -0.6em \intop}\nolimits_{#1}}%
          {\mathop{\kern 0.1em\vrule width 0.5em height 0.69678ex depth -0.60387ex
                  \kern -0.6em \intop}\nolimits_{#1}}%
          {\mathop{\kern 0.1em\vrule width 0.5em height 0.69678ex depth -0.60387ex
                  \kern -0.6em \intop}\nolimits_{#1}}}

\usepackage[colorlinks=true, linkcolor=blue, citecolor=black]{hyperref}

\usepackage{enumitem}

\addto\extrasenglish{}
\addto\extrasenglish{}

\allowdisplaybreaks

\swapnumbers

\renewcommand{\d}{\textnormal{d}}
\newcommand{\eps}{\varepsilon}

\newcommand{\WspOm}{W^{s,p}(\Omega|\Omega^c)}
\newcommand{\VsOm}{V^s(\Omega|\R^d)}

\usepackage[backgroundcolor=white, bordercolor=blue,
linecolor=blue]{todonotes}

\title{Function spaces and extension results for nonlocal Dirichlet problems}

\subjclass[2010]{46E35, 47G20, 35D30, 35S15}
\keywords{Nonlocal Dirichlet problem, trace theorem, extension 
operator, integrodifferential equation}

\author[B{.} Dyda]{Bart{\l}omiej Dyda}
\author[M{.} Kassmann]{Moritz Kassmann}

\address[B.D.]{Faculty of Pure and Applied Mathematics\\ Wroc{\l}aw University 
of Science and Technology\\
Wybrze\.ze Wyspia\'nskiego 27,
50-370 Wroc{\l}aw, Poland
}
\email{bdyda@pwr.edu.pl}

\address[M.K.]{Fakult\"{a}t f\"{u}r Mathematik\\Universit\"{a}t 
Bielefeld\\Postfach 100131\\D-33501 Bielefeld}
\email{moritz.kassmann@uni-bielefeld.de}
\thanks{B.D. was partially supported by grant NCN 2015/18/E/ST1/00239.}
\thanks{B.D. acknowledges support by DFG-Sonderforschungsbereich 701.} 

\begin{document}

\begin{abstract}
We study function spaces and extension results in relation with Dirichlet 
problems involving integrodifferential operators. For such problems, data 
are prescribed on the complement of a given domain $\Omega \subset \R^d$. 
We introduce a function space that serves as a trace space for nonlocal 
Dirichlet problems and study related extension results.
\end{abstract}

\maketitle

\section{Introduction}

In this work, we study function spaces related to Dirichlet problems for a class 
of integrodifferential 
operators, which satisfy the maximum principle. We introduce a new function 
space, which can be understood as a nonlocal trace space. Let us illustrate our 
task with a very 
simple problem. Let $\Omega = B_1 \subset \R^d$ be the unit ball and assume 
$0 < s < 1$. We ask ourselves the question, for which 
functions $g : \R^d \setminus \Omega \to \R$, there is a function $u:\R^d \to 
\R$ satisfying 

\begin{alignat}{2}
\lim\limits_{\eps \to 0} \int\limits_{\R^d \setminus B_\varepsilon} 
\frac{u(x\!+\!h) \!-\! u(x)}{|h|^{d+2s}}\,
\mathrm{d} h &= 0 &&\text{ for } x \in \Omega\,, \label{eq:Dprobeqn}\\
 u(x) &= g(x) \quad &&\text{ for } x \in \R^d \setminus 
\Omega\,. \label{eq:Dprobdata}  
\end{alignat}

Note that \eqref{eq:Dprobeqn} is equivalent to $(-\Delta)^s u = 0$ in $\Omega$. 
In order to discuss the possible choices of data $g$, we need to specify the 
function space of possible solutions $u$. Moreover, we 
have to explain in which sense the above equation is 
to be understood. Since the validity of \eqref{eq:Dprobeqn} for some 
$x \in 
\Omega$ involves values of $u$ on $\R^d \setminus \Omega$, where $u=g$ is 
imposed, there is a direct link between the function 
space for solutions $u$ and the function space for the data $g$. 

\medskip

The set-up of boundary value problems is well understood for differential 
operators, i.e., in the limit case $s=1$. However, by considering our 
results for $s \to 1^{-}$, we will obtain a new extension result for classical 
Sobolev 
spaces, cf. \autoref{cor:limit_case-example} and \autoref{cor:limit_case}. 

\medskip

Let us explain how to define a variational solution $u$ satisfying 
\eqref{eq:Dprobeqn}--\eqref{eq:Dprobdata}, cf. \cite{DRV14, FKV15}. Define two 
vector spaces by 
\begin{align*}
\VsOm &= \big\{ v \in L^2_{\operatorname{loc}}(\R^d) | 
\int\limits_\Omega \int\limits_{\R^d} \frac{\big(v(y) - v(x) 
\big)^2}{|x-y|^{d+2s}} \d x \, \d y < 
\infty \big\}, \\
H^s_\Omega(\R^d) &= V^s_0(\Omega|\R^d) = \{v \in V^s(\Omega|\R^d)| v= 0 \text{ 
on } \Omega^c \}\,. 
\end{align*}
Let us collect a few basic observations on these spaces.
\begin{enumerate}
 \item $V^s(\R^d|\R^d)$ and $V^s(\R^d|\R^d) \cap L^2(\R^d)$ equal the Sobolev-Slobodeckij space 
$\dot{H}^s(\R^d)$ and $H^s(\R^d)$, respectively.
\item $H^s_\Omega(\R^d)$ is a Banach space together with the norm
\begin{align*}
\|v\|_{H^s}^2 = \|v\|^2_{L^2} + (1\!-\!s) \int\limits_{\R^d} \int\limits_{\R^d} 
\frac{\big(v(y) - v(x) 
\big)^2}{|x-y|^{d+2s}} \d y \d x 
\end{align*}
\end{enumerate}

Let us define the notion of a variational solution. 

\begin{defn}[cf. Definition 2.5 in \cite{FKV15}]\label{def:var-sol}
Let $\Omega \subset \R^d$ be open such that $\Omega$ and $\Omega^c$ both have 
positive measure. Let $g \in \VsOm$. Then $u \in \VsOm$ is 
called a 
variational solution to \eqref{eq:Dprobeqn}--\eqref{eq:Dprobdata}, if $u-g \in 
H^s_\Omega(\R^d)$ and for every $\varphi \in H^s_\Omega(\R^d)$ 
\begin{align}
\int\limits_{\R^d} \int\limits_{\R^d} \frac{\big(u(y) - u(x) 
\big) \big(\varphi(y) - \varphi(x) \big) }{|x-y|^{d+2s}} \d y \d x = 0 \,. 
\end{align}
\end{defn}

\begin{rem} (i) The above definition implies that solutions $u$ belong to 
$L^2(\R^d,\mathrm{d}m)$ with $m(\mathrm{d}x) = 
(1+|x|)^{-d-2s} \, \mathrm{d}x$. It would be possible to work under 
the assumption $u \in L^1(\R^d, \mathrm{d}m)$, but the presentation would be 
less transparent. (ii) In peridynamics, the definition of variational solutions 
to nonlocal boundary value problems looks similar, cf. \cite{DuMe14}. However, 
it is rather different because of the usage of more restrictive 
function spaces. Regularity of $u$ respectively $g$ is required in regions, 
which are away from that region, where the nonlocal equation is considered. The 
above definition avoids such an assumption.    
\end{rem}

With the above 
definition at hand, we are now in the position to explain the main question 
addressed in this 
article. In order to apply \autoref{def:var-sol} one needs to prescribe the 
data 
function $g$ in the vector space $\VsOm$, i.e. in particular one needs 
to 
prescribe all values of $g$ in $\R^d$. This leads to the following question: 

\medskip

\begin{tabbing}
{\bf Question:} \= For which Banach space of functions $g: \Omega^c \to \R$ \\ 
\> (a) is there an extension operator $g \mapsto \ext(g) \in \VsOm$, and \\ 
\> (b) is there a trace operator from $\VsOm$ into this 
space? 
\end{tabbing}

\medskip

Extension and trace theorems are well known in the study of classical 
local Dirichlet problems. Thus, for the case of Sobolev spaces of integer 
order, these questions are classical and answers were given long time 
ago , cf. 
\cite{Slo58} for an early work and \cite{AdFo03} for a general exposition. For 
a large class of domains $\Omega$, functions in $H^{1/2}(\partial \Omega)$ 
can be extended to elements of $H^1(\Omega)$ and these themselves have a trace 
in $H^{1/2}(\partial \Omega)$. A side result of our research on nonlocal 
quantities is that instead of $H^1(\Omega)$ one could also consider the 
much larger space of all $L^2(\Omega)$-functions $v$ with 

\begin{align}
\int_{\Omega} \int_{\Omega} 
\frac{|v(x)-v(y)|^2}{(|x-y|+\delta_x+\delta_y)^{d+2}}\,dy\,dx < \infty \,, 
\end{align}
where $\delta_z = \operatorname{dist}(z, \partial \Omega)$ 
for $z \in \R^d$. 

\medskip

Trace and extension results have been established for various function spaces 
including Sobolev spaces with fractional order of differentiability. 
To our best knowledge, extensions from the complement of a domain to the whole 
space have not been dealt with so far. One reason for this might be that 
Dirichlet problems with prescribed data on the complement have not yet been 
studied intensively.  

\medskip

Let us formulate our main result, which answers the aforementioned question. We allow the domain $\Omega$ to have a rather rough 
boundary, but we stress the fact that our results are new even for domains $\Omega$ with a smooth boundary. See \autoref{sec:prelims} for the definition of inner respectively exterior thickness of domains. Note that any bounded Lipschitz domain has these properties.  The inner radius of an open set $\DD\subset \R^d$ is defined as 
$\inr(\DD)=\frac12 \sup_{B\subset \DD} \diam(B)$, where the supremum is taken 
over all 
balls $B\subset \DD$.

\begin{thm}\label{thm:answer}
Assume $0 < s < 1$. Let $\Omega\subset \R^d$ be open, interior thick and 
exterior thick such 
that $\partial \Omega$ has Lebesgue measure zero,  and $\inr(\Omega)<\infty$ or 
$\inr(\Omega^c)=\infty$. Then the following is true:
\begin{enumerate}
\item[(a)]
  If $f\in L^p_{loc}(\R^d)$, $1 \leq p < \infty$, satisfies
  \begin{equation}\label{e:space}
    \int_{\Omega} \int_{\R^d}  \frac{|f(x)-f(y)|^p}{|x-y|^{d+sp}}\,dy\,dx < 
\infty,
  \end{equation}
  then
  \begin{equation}\label{e:spacecomp}  
    \int_{\Omega^c} \int_{\Omega^c} 
    \frac{|f(x)-f(y)|^p}{(|x-y|+\delta_x+\delta_y)^{d+sp}}\,dy\,dx < \infty \,.
  \end{equation}
\item[(b)] There exists a linear extension operator $\ext$, which maps 
$L^p_{loc}(\Omega^c)$, $1 \leq p < \infty$, to measurable functions defined on 
$\R^d$ such that 

\begin{align}
(1-s) \int_{\Omega} \int_{\R^d}  
\frac{|\ext(f)(x)-\ext(f)(y)|^p}{|x-y|^{d+sp}}\,dy\,dx \asymp \int_{\Omega^c} 
\int_{\Omega^c} 
    \frac{|f(x)-f(y)|^p}{(|x-y|+\delta_x+\delta_y)^{d+sp}}\,dy\, dx \,,
\end{align}
with constants that depend only on $\inf s$, $p$, $d$ and $\Omega$.
\end{enumerate}
\end{thm}

\autoref{thm:answer} follows directly from \autoref{thm:norm} and 
\autoref{thm:extension}. 

\medskip

Considering the limit $s \to 1^{-}$, \autoref{thm:answer} implies a new 
extension-type result for classical Sobolev spaces. We formulate this 
observation in the special case $\Omega = B_1 \subset \R^d$ and refer to 
\autoref{cor:limit_case} for the general case and to  \autoref{rem:Du} for some 
related result. 

\begin{cor}\label{cor:limit_case-example}
Given $1 < p < \infty$, there is a constant $c = c(d,p) \geq 1$ such that 
\[
 \int \limits_{B_1} |\nabla \ext(f)|^p \leq c \int\limits_{B_2\setminus B_1}\! 
\int\limits_{B_2\setminus B_1}
 \frac{|f(x)-f(y)|^p}{(|x-y|+\delta_x+\delta_y)^{d+p}}~dx\,dy
\]
for every $f \in L^p(B_2\setminus B_1)$ such that the right-hand side is finite.
\end{cor}

\autoref{cor:limit_case-example} is a special case of \autoref{cor:limit_case}.

\medskip

The article is organized a follows. In \autoref{sec:setup_results} we present 
the setup of our work together with the main results, \autoref{thm:norm} and 
\autoref{thm:extension}. \autoref{sec:prelims} 
provides basic properties of the function spaces under consideration. In 
\autoref{sec:proof1} we present the proof of \autoref{thm:norm}. The proof of 
\autoref{thm:extension} is given in \autoref{sec:proof2}.

\section{Setup and detailed results}\label{sec:setup_results}

Throughout the whole paper we assume that 
$\Omega\subset\R^d$ is an open set with the property that both, $\Omega$ and 
$\Omega^c = \R^d \setminus \Omega$, have positive Lebesgue measure. For our 
main result, we will assume some very mild additional assumption. We will use 
the symbol $g \lesssim h$ to denote that the inequality $g \leq c 
h$ holds with a positive constant $c$ that is independent of $g$ and 
$h$. We adopt the convention that $0^a=\infty$ for $a<0$, in particular, 
$\frac{1}{0}=\infty$. We assume $0<p<\infty$ and $0<s\leq 1$.

\medskip

In short, our main result answers the question from the previous section. It 
roughly says that the vector space of all functions $g \in 
L^2_{loc}(\Omega^c)$ with
\begin{align}
\int_{\Omega^c} \int_{\Omega^c} 
\frac{|g(x)-g(y)|^2}{(|x-y|+\delta_x+\delta_y)^{d+2s}}\,dy\,dx < \infty
\end{align}
has the desired properties, see \autoref{thm:answer}. A special feature of our 
result is that the limit 
case $s=1$ can be included. Thus we obtain a new extension result for 
$W^{1,2}(\Omega)$-functions, see below for details.

Let us now explain the set-up in detail. For $f\in L^p(\R^d)$, define
\begin{align}
 |f|_{\WspOm}^p &:= \int_\Omega \int_{\Omega^c} 
\frac{|f(x)-f(y)|^p}{|x-y|^{d+sp}}\,dy\,dx, \label{eq:WspOm-seminorm}\\
 \|f\|_{\WspOm}^p &:= \|f\|_{L^p(\R^d)}^p +  |f|_{\WspOm}^p, 
\label{eq:WspOm-norm}
\end{align}
and let $\WspOm = \{ f\in L^p(\R^d): \|f\|_{\WspOm} < \infty \}$. If $f 
\in L^p(\R^d)$, then $f \in \WspOm$, if $f$ satisfies some regularity 
condition across the boundary $\partial \Omega$, whereas the behavior of $f$ 
far from $\partial \Omega$ is not considered.

\begin{example*}
Consider a bounded Lipschitz domain $\Omega \subset \R^d$ and $f \in 
L^p(\R^d)$, $1 \leq p < \infty$, given by 
$f = \mathbbm{1}_\Omega$. Then the function $f$ belongs to 
$W^{s,p}(\Omega|\Omega^c)$ if and only if $s < \frac{1}{p}$.
\end{example*}

Recall that the inner radius of an open set $\DD\subset \R^d$ is defined as 
$\inr(\DD)=\frac12 \sup_{B\subset \DD} \diam(B)$, where the supremum is taken 
over all 
balls $B\subset \DD$. For $x\in \R^d$, set $\delta_x = \dist(x,\Omega)$. 
For $0<\delta,\varepsilon \leq \infty$ set 
\[
\Omint_\delta = \{x\in\Omega : \dist(x,\Omega^c) \leq \delta\}, \qquad
\Omext_\varepsilon = \{x\in\Omega^c : \dist(x,\Omega) < \varepsilon \}.
\]
Note that $\Omint_{\inr(\Omega)} = \Omega$.
For a function $g$ 
let 
\[
 |g|_{A,B}^{s,p} :=  \int_A \int_{B} 
\frac{|g(x)-g(y)|^p}{(|x-y|+\delta_x+\delta_y)^{d+sp}}\,dy\,dx.
\]

\medskip

The following result introduces a useful (semi)norm that is equivalent to 
$\|f\|_{\WspOm}$ respectively $|f|_{\WspOm}$. For the definition of interior thick 
domains we refer the reader to  \autoref{subsec:whitney}, here let us only mention 
that bounded Lipschitz domains are interior thick.

\begin{thm}\label{thm:norm}
Let $0<p<\infty$ and $0<s\leq 1$.
Suppose that $\Omega\subset \R^d$ is an open interior thick set.
Then there exists a constant $c=c(p,\Omega)$ not depending on $s$, such that
\begin{equation}\label{eq:comp-semi}
c^{-1} |f|_{\WspOm}^p \leq 
  \int_{ \Omega \cup \Omext_{\inr(\Omega)} } \int_{\Omega^c} 
\frac{|f(x)-f(y)|^p}{(|x-y|+\delta_x+\delta_y)^{d+sp}}\,dy\,dx
\leq \frac{c}{s} |f|_{\WspOm}^p
\end{equation}
for every $f\in L^p(\R^d)$. The following norms
  \[
  \left(  \| \cdot \|_{L^p(\Omega, (1+|x|)^{-d-sp}\,dx)}^p + |\cdot|_{\WspOm}^p 
\right)^{1/p},
  \]
  \[
  \left(  \| \cdot \|_{L^p(\R^d, (1+|x|)^{-d-sp}\,dx)}^p + |\cdot|_{\WspOm}^p 
\right)^{1/p} \,,
  \]
  and
  \[
  \left( \|f\|_{L^p(\R^d, (1+|x|)^{-d-sp}\,dx)}^p + 
      \int_{ \R^d } \int_{\Omega^c} 
\frac{|f(x)-f(y)|^p}{(|x-y|+\delta_x+\delta_y)^{d+sp}}\,dy\,dx \right)^{1/p}  
      \]
are comparable with constants depending only on $p$, $\Omega$ and the lower 
bound for $s$.
\end{thm}

\begin{rem}
Note that, for the case $s \to 1^{-}$, the different $s$-dependence on the two 
sides in \eqref{eq:comp-semi} is not important.
\end{rem}

\begin{example} Define $f: \R^d \to \R$ by $f(x) = 
\sqrt{|x|-1}$ for $1 < |x| < 2$ and $f(x) = 0$ elsewhere. Assume $\Omega = B_1 
\subset \R^d$ as in \autoref{cor:limit_case-example}. Then both expressions, 
\begin{align*}
\int_{ \Omega \cup \Omext_{\inr(\Omega)} } \int_{\Omega^c} 
\frac{|f(x)-f(y)|^2}{(|x-y|+\delta_x+\delta_y)^{d+2s}}\,dy\,dx 
\quad \text{ and } \quad |f|_{W^{s,2} (\Omega|\Omega^c)}^2
\end{align*}
diverge for $s \to 1^{-}$. As observed in \cite[Sec. 2.2.4]{PV-thesis}, the 
expression $(1-s) |f|_{W^{s,2}(\Omega|\Omega^c)}$ remains bounded. 
\end{example}

\medskip

The following theorem contains our main result.

\begin{thm}\label{thm:extension}
Let $\Omega\subset \R^d$ be an open set which is exterior thick and such that  
$\partial \Omega$ has Lebesgue measure zero,
and $\inr(\Omega)<\infty$ or $\inr(\Omega^c)=\infty$. Then there exists 
a~linear operator $\ext$ which maps $L^1_{loc}(\Omega^c)$ to 
the space of measurable functions on $\R^d$ with the following properties.
\begin{enumerate}
\item[(a)]
For all $f\in L^1_{loc}(\Omega^c)$, $\ext(f)|_{\Omega^c} = f$ and 
$\ext(f)|_\Omega \in C^\infty(\Omega)$.
Furthermore, if $z_0\in \partial \Omega$ and the limit
$ g=\lim_{\Omega^c\ni x\to z_0} f(x)$
exists, then also the limit $\lim_{\Omega\ni x\to z_0} \ext(f)(x)$ exists and 
equals $g$.

\item[(b)]
Let $1\leq p<\infty$.
There exists a~constant $c=c(\Omega,p)$ such that the following inequalities 
hold
for all $f\in L^1_{loc}(\Omega^c)$ and $0<\delta\leq \varepsilon \leq \infty$
\begin{align}
 |\ext(f)|_{\Omint_\delta, \Omext_\varepsilon}^{s,p} 
&\leq \frac{c}{s} |f|_{\Omext_\delta,\Omext_\varepsilon}^{s,p}, \quad 0<s\leq 1, \label{e:int-ext}\\
 |\ext(f)|_{\Omint_\delta, \Omint_\delta}^{s,p} 
 &\leq \frac{c}{s(1-s)} |f|_{\Omext_\delta, \Omext_\delta}^{s,p}, \quad 0<s<1. \label{e:int-int}
\end{align}
In particular,
\begin{align}
 |\ext(f)|_{\Omega, \Omega^c}^{s,p} 
&\leq \frac{c}{s} |f|_{\Omega^c,\Omega^c}^{s,p}, \quad 0<s\leq 1, \label{e:simpl1}\\
 |\ext(f)|_{\R^d, \R^d}^{s,p} 
&\leq \frac{c}{s(1-s)} |f|_{\Omega^c, \Omega^c}^{s,p}, \quad 0<s<1. \label{e:simpl2} 
\end{align}

\item[(c)]
Let $1\leq p < \infty, \beta\in \R$ or $p = \infty, \beta =0$. There 
exists a~constant 
$c=c(\Omega,\beta, p)$ such that the following inequality holds
for all $f\in L^1_{loc}(\Omega^c)$
\[
 \| \ext(f)\|_{L^p(\Omega, (1+|x|)^\beta dx)} \leq c \| f \|_{L^p(\Omext_{\inr(\Omega)},  
(1+|x|)^\beta dx)}.
\]
\end{enumerate}
\end{thm}

From \autoref{thm:norm} and \autoref{thm:extension}, the answer
to the question posed earlier immediately follows, cf. \autoref{thm:answer}.

\begin{cor}\label{cor:limit_case}
Let $\Omega$ be a bounded Lipschitz-domain and $1<p<\infty$. Then there 
exists a~constant $c=c(\Omega,p)$ such that \begin{equation}\label{e:extW1p}
 |\ext(f)|_{W^{1,p}(\Omega)} \leq c 
|f|_{\Omext_{\inr(\Omega)},\Omext_{\inr(\Omega)}}^{1,p}, \quad f\in 
L^1_{loc}(\Omega^c),
\end{equation}
where we take $|\ext(f)|_{W^{1,p}(\Omega)} = \|\nabla \ext(f)\|_{L^p(\Omega)}$, 
if $\ext(f)\in W^{1,p}(\Omega)$, and $|\ext(f)|_{W^{1,p}(\Omega)}=\infty$ 
otherwise.
\end{cor}

\begin{proof}
We put $\delta=\varepsilon=\inr(\Omega)$ in \eqref{e:int-int},
multiply its both sides by $(1-s)$ and take $s\to 
1^{-}$. If $|f|_{\Omext_{\inr(\Omega)}, {\Omext_{\inr(\Omega)}}}^{s,p}<\infty$ for 
some~$s$, 
then $f\in L^p(\Omega)$ and inequality \eqref{e:extW1p} follows from 
\cite[Theorem 2]{BBM}. In the other case inequality \eqref{e:extW1p} is trivial.
\end{proof}

\begin{rem}\label{rem:Du}
In the case $p=2$, a result related to \autoref{cor:limit_case} has 
recently 
been established in \cite{DuTi16}. For the trace map $T$, the authors prove an 
estimate of the form
\begin{align}
\|Tu\|_{H^{1/2}(\partial \Omega)} \leq C \|u\|_{\mathcal{S}(\Omega)} \,, 
\end{align}
where 
\begin{align*}
\|u\|^2_{\mathcal{S}(\Omega)} = \|u\|^2_{L^2(\Omega)} + \int\limits_{\Omega} 
\int\limits_{\Omega \cap B(x,\delta(x))} \frac{\big(u(y) - 
u(x)\big)^2}{\delta(x)^{d+2}} \mathbbm{1}_{B_1}(|y-x|) dy \, dx \,,
\end{align*}
and $\delta$ denotes the distance function with respect to $\partial \Omega$. 
The authors of \cite{DuTi16} are interested in models from 
peridynamics. It is interesting that 
our approach to nonlocal function spaces, in the limit case $s \to 1^{-}$, leads 
to 
a similar nonlocal trace theorem as their approach. Note that \cite{DuTi16} 
does not contain extension results like \autoref{thm:extension}.
\end{rem}

\section{Preliminary results}\label{sec:prelims}

In this section, we prove basic properties of the function spaces $\WspOm$ and 
collect several result on inner thick respectively exterior thick domains. 

\subsection{Basic properties \texorpdfstring{of $\WspOm$}{}}

Recall the definitions from \eqref{eq:WspOm-seminorm} and \eqref{eq:WspOm-norm}.

\begin{prop}
  Let $\Omega\subset \R^d$ be an open set, $0<s<1$ and $p\geq 1$.
Then the space $\WspOm$ equipped with the norm $\|\cdot \|_{\WspOm}$ is 
a~Banach space.
\end{prop}
\begin{proof}
   The proof is straightforward. Let $(f_n)$ be a~Cauchy sequence in $(\WspOm,  
\|\cdot \|_{\WspOm} )$.
  Then $(f_n)$ is a~Cauchy sequence in $L^p(\R^d)$, hence there exists a 
function $f \in 
L^p(\R^d)$ such that $f_n\to f$ in $L^p(\R^d)$.
  Let $f_{k_n}$ be a~subsequence convergent a.e. to  $f$.
  By the Fatou lemma
  \begin{align*}
  |f_{k_n}-f|_{\WspOm}^p &=
  \int_\Omega \int_{\Omega^c} \liminf_{l \to\infty} 
\frac{|(f_{k_n}-f_l)(x)-(f_{k_n}-f_l)(y)|^p}{|x-y|^{d+sp}}\,dy\,dx \\
  &\leq \liminf_{l \to\infty}  \int_\Omega \int_{\Omega^c}  
\frac{|(f_{k_n}-f_l)(x)-(f_{k_n}-f_l)(y)|^p}{|x-y|^{d+sp}}\,dy\,dx \to 0
  \text{ as } n\to \infty\,.
  \end{align*}
  From the above calculation and triangle inequality we deduce that $f\in 
\WspOm$.
  Since $(f_n)$ is a~Cauchy sequence in $(\WspOm,  \|\cdot 
\|_{\WspOm} )$ and its subsequence converges to $f$, the whole 
sequence converges to $f$.
\end{proof}

\begin{rem}
  For $p\in (0,1)$ the space $\WspOm$ equipped with a~metric 
$\rho(f,g):=\|f-g\|_{\WspOm}^p$ is complete. The proof is 
basically the same as above.
  \end{rem}

\begin{prop}
  If a~measurable function $f:\R^d\to\R$ satisfies $|f|_{\WspOm} < \infty$, then
  $f\in L^p(\R^d, (1+|x|)^{-d-sp}\,dx)$.
  Furthermore, the norms
  \[
  \left(  \| \cdot \|_{L^p(\Omega, (1+|x|)^{-d-sp}\,dx)}^p + |\cdot|_{\WspOm}^p 
\right)^{1/p},
  \]
  \[
  \left(  \| \cdot \|_{L^p(\R^d, (1+|x|)^{-d-sp}\,dx)}^p + |\cdot|_{\WspOm}^p 
\right)^{1/p}
  \]
are comparable.
\end{prop}
\begin{proof}
  Let $R>1$ be large enough so that $B(0,R)$ intersects both $\Omega$ and 
$\interior \Omega^c$. For a~given $f$ as in the proposition, let $n \in \N$ be 
such 
that for $E_n=\{x\in \R^d: |f(x)|\leq n\}$ the intersections $F_n = E_n \cap 
\Omega\cap B(0,R)$ and $G_n = E_n \cap \Omega^c\cap B(0,R)$ are of positive 
Lebesgue measure. Note that
  \[
  |w-z| \leq R+|z| \leq R(1+|z|), \qquad \text{for $w\in B(0,R)$ and $z\in 
\R^d$.}
  \]
  Therefore
  \begin{align*}
    2|f|_{\WspOm}
    &\geq
    \int_{F_n} \int_{\Omega^c\setminus E_{2n}}  
\frac{|f(x)-f(y)|^p}{|x-y|^{d+sp}}\,dy\,dx
    + \int_{\Omega \setminus E_{2n}} \int_{G_n} 
\frac{|f(x)-f(y)|^p}{|x-y|^{d+sp}}\,dy\,dx \\
    &\geq \frac{2^{-p}}{R^{d+sp}} \left(
    \int_{F_n} \int_{\Omega^c\setminus E_{2n}} \frac{|f(y)|^p}{(1+|y|)^{d+sp}} 
\,dy 
\,dx + \int_{\Omega \setminus E_{2n}} \int_{G_n} 
\frac{|f(x)|^p}{(1+|x|)^{d+sp}} \,dy\,dx \right) \\
    &\geq
    \frac{2^{-p}(|{F_n}|\wedge |{G_n}|)}{R^{d+sp}}
    \int_{\R^d\setminus E_{2n}} \frac{|f(x)|^p}{(1+|x|)^{d+sp}} \,dx.
  \end{align*}
Choose $n \in \N$ sufficiently large so that $|{F_n}|\wedge |{G_n}|$ is 
positive. Since obviously $\int_{E_{2n}} |f(x)|^p (1+|x|)^{-d-sp} \,dx < 
\infty$, we 
conclude that $f\in L^p(\R^d, (1+|x|)^{-d-sp}\,dx)$. 

  Comparability of the first two norms follows from the following inequalities
  \begin{align*}
    \| f \|_{L^p(\Omega, (1+|x|)^{-d-sp}\,dx)}^p + |f|_{\WspOm}^p
    &
    \gtrsim 
    \int_{\Omega\cap B(0,R)} \left( |f(x)|^p + \int_{\Omega^c} 
\frac{|f(x)-f(y)|^p}{(1+|y|)^{d+sp}}  \,dy \right) \,dx \\
    &\gtrsim 
    \int_{\Omega\cap B(0,R)} \int_{\Omega^c} 
\frac{|f(x)|^p+|f(x)-f(y)|^p}{(1+|y|)^{d+sp}}  \,dy  \,dx \\
    &\gtrsim
     \int_{\Omega^c} \frac{|f(y)|^p}{(1+|y|)^{d+sp}}  \,dy 
  \end{align*}
  with constants depending only on $\Omega$, $R$, $d$, $s$, $p$.
\end{proof}

\subsection{Whitney decomposition, thickness and plumpness}\label{subsec:whitney}
We recall  several geometric notions needed in the 
sequel. They allow us to present our main results for rather general 
domains $\Omega \subset \R^d$. Note that, however, \autoref{thm:norm} and 
\autoref{thm:extension} are new even for domains with a smooth boundary.

\medskip

For a nonempty open set $\DD\subset \R^d$, $\DD\neq \R^d$, we fix 
a Whitney decomposition $\wh(\DD)$ \cite[VI.1]{MR0290095} and write
$\wh_m(\DD)$
for the family of Whitney cubes with side length $2^{-m}$, $m\in\Z$.
If $Q\in\wh(\DD)$, then
\begin{equation}\label{dist_est}
  \diam(Q)\le \dist(Q,\partial \DD)\le 4\diam(Q)\,.
\end{equation}
For any cube $Q$, its side length is denoted by $\ell(Q)$ and its center by 
$x_Q$.
By $Q^*$ we denote a cube with the same center as $Q$, but side length 
$\ell(Q^*)=(1+\varepsilon)\ell(Q)$, where $0<\varepsilon<1/4$ is fixed once for 
all.
Such cubes have the property that 
\[
 \ind_D \leq \sum_{Q\in \wh(\DD)} \ind_{Q^*} \leq M \ind_D
\]
 with some constant $M$ depending only on $d$.

The next two definitions are slightly modified versions of \cite[Definition 
3.1]{Triebel2008}.
Our definitions and  \cite[Definition 3.1]{Triebel2008} coincide if $\DD$ or 
$\DD^c$ has finite inner radius.
In the case when $\inr(\DD)=\inr(\DD^c)=\infty$, if the domain $\DD$ is 
$I$-thick in the sense of \autoref{def:Ithick}, then it is also 
$I$-thick 
in the sense of \cite[Definition 3.1]{Triebel2008}.

\begin{defn}\label{def:Ithick}
An open set $\DD\subset \R^d$ is called \emph{$I$-thick} (\emph{interior 
thick}), if
for every $M>0$ there exists a constant $C$ such that for every cube $Q\in 
\wh(\R^d\setminus \overline{\DD})$ with $\diam Q < M \inr(\DD)$ there exists
a~\emph{reflected} cube $\widetilde{Q} \in \wh(\DD)$ satisfying
\begin{equation}\label{e:Ithick}
 C^{-1} \diam(Q) \leq \diam(\widetilde{Q}) \leq C \diam(Q) \quad\text{and}\quad
 \dist(\widetilde{Q}, Q) \leq C \dist(Q,\partial \DD).
\end{equation}
\end{defn}

\begin{defn}\label{def:Ethick}
An open set  $\DD\subset \R^d$ is called \emph{$E$-thick} (\emph{exterior 
thick}), if
for every $M>0$ there exists a constant $C$ such that for every cube $Q\in 
\wh(\DD)$ with $\diam Q < M \inr(\DD^c)$ there exists
a~\emph{reflected} cube $\widetilde{Q} \in \wh(\R^d\setminus\overline{\DD})$ 
satisfying
\begin{equation}\label{e:Ethick}
 C^{-1} \diam(Q) \leq \diam(\widetilde{Q}) \leq C \diam(Q) \quad\text{and}\quad
 \dist(\widetilde{Q}, Q) \leq C \dist(Q,\partial \DD).
\end{equation}
\end{defn}

\begin{rem}
The definitions of $I$- and $E$-thickness do not depend on the choice of the 
families of Whitney cubes $\wh(\DD)$ and $\wh(\R^d\setminus \overline{\DD})$.
\end{rem}

\begin{rem}\label{rem:lambda}
  Let $\lambda>0$ be fixed.
In \autoref{def:Ithick} we may additionally assume that the reflected 
cubes satisfy
\begin{equation}\label{e:Ithicklambda}
 \diam \widetilde{Q} \leq \lambda \diam Q.
\end{equation}
Indeed, if the opposite inequality holds, then
\[
\dist(\widetilde{Q}, \partial \DD) \geq \diam(\widetilde{Q}) > \lambda \diam Q,
\]
 so in the ball
$B(x_{\widetilde{Q}}, 5\diam \widetilde{Q})$ there exists a~point $z\in \DD$ 
with 
$\dist(z,\partial \DD) = \lambda \diam Q$.
Take $Q'$ to be a~cube from $\wh(\DD)$ containing $z$. Then
\[
\frac{\lambda}{5} \diam Q \leq \frac15 \dist(z,\partial \DD) \leq \diam Q'\leq 
\dist(z,\partial \DD) = \lambda \diam Q,
\]
so $\diam Q$ and $\diam Q'$ are comparable. Moreover, since $z\in 
B(x_{\widetilde{Q}}, 5\diam \widetilde{Q}) \cap Q'$,
we obtain
\[
\dist(Q', \widetilde{Q}) \leq \dist(z,x_{\widetilde{Q}}) \leq 
5\diam\widetilde{Q},
\]
therefore,
\[
\dist(Q', Q) \leq \diam Q' + \dist(Q', \widetilde{Q}) + \diam\widetilde{Q} + 
\dist(\widetilde{Q}, Q) \lesssim \dist(Q,\partial \DD)
\]
with a constant depending only on $\lambda$ and $C$.
Consequently, 
\eqref{e:Ithick} holds also for $Q'$ in place of $\widetilde{Q}$ (perhaps with 
an 
enlarged $C$).
Hence by redefining reflected cubes both \eqref{e:Ithick} and 
\eqref{e:Ithicklambda} hold.

A similar remark applies to \autoref{def:Ethick}.
\end{rem}

\begin{rem}\label{rem:Ithickmod}
In \autoref{def:Ithick} we may additionally assume that the reflected 
cubes satisfy
\begin{equation}\label{e:Ithickmod}
 \widetilde{Q}\subset \{x\in \DD: \dist(x,\partial \DD) < \inr(\DD^c) \}.
\end{equation}
Indeed, by taking $\lambda \leq\frac15$ in \autoref{rem:lambda} and perhaps 
redefining reflected cubes, we obtain
$\dist(w,\partial \DD) \leq 5\diam \widetilde{Q} \leq \diam Q < \inr(\DD^c)$ 
for 
$w\in \widetilde{Q}$, as desired.

A similar remark applies to \autoref{def:Ethick}.
\end{rem}

\begin{rem}\label{r:overlap}
Let $D$ be exterior thick.
The family of all reflected cubes in the sense of the definition above, i.e., 
$\mathcal{F}:=\{\widetilde{Q} : Q\in 
\wh(\DD),\ \diam Q < M \inr(\DD^c)\}$ has the bounded overlap property, i.e.,
there exists a~constant $N$ such that
\[
\sum_{Q\in \wh(\DD)} \ind_{\widetilde{Q}} \leq N\ind_{\R^d\setminus 
\overline{\DD}}.
\]
This estimate holds true because the size of $\widetilde{Q}$ and its distance 
to 
$Q$ are comparable to the 
size of $Q$. An analogous property holds for interior thick sets $D$.
\end{rem}

From  \cite[Proposition 3.6]{Triebel2008} it follows that if $\DD$ is a bounded 
$(\epsilon,\delta)$-domain \cite[Definition 3.1(i)]{Triebel2008}, then $\DD$ is 
$I$-thick and $\partial \DD$ has Lebesgue measure zero. Bounded Lipschitz 
domains are both $I$-thick and $E$-thick \cite[Proposition 3.8]{Triebel2008}.

We will show that the assumption that $\DD$ is an $(\epsilon,\delta)$-domain 
may 
be replaced by a~weaker one. To this end we need the following definition.

\begin{defn}\cite{MR927080, mv}
A set $A\subset \R^d$ is {\em $\kappa$-plump}
with $\kappa\in (0,1)$ (or simply \emph{plump}) if, for each $0<r< 
\mathrm{diam}(A)$ and each $x\in \bar{A}$, there
is $z\in \overline{B(x,r)}$ such that
$B(z,\kappa r)\subset A$.
\end{defn}

\begin{lem}\label{l.plump}
If $\DD\subset \R^d$ is plump, then it is also $I$-thick and $\partial \DD$ has 
$d$-dimensional  Lebesgue measure zero.
\end{lem}

\begin{proof}
Let us note that if $\DD$ is plump, then its boundary $\partial \DD$ is 
\emph{porous},
 i.e., there exists a~constant $\alpha$
with the following property: for every $x\in \R^d$ and $0<r\leq 1$, there 
exists 
$y\in B(x,r)$ 
such that $B(y,\alpha r) \subset B(x,r) \setminus \partial \DD$.
Therefore $\partial \DD$ has Lebesgue measure zero, see e.g. \cite{Luukkainen}.

Let $M>1$.
For each cube $Q\in \wh(\interior \DD^c)$ such that $\diam Q < M \inr(\DD)$ we 
will associate a~\emph{reflected} cube $\widetilde{Q} \in  \wh(\DD)$ in the 
following way. Let $y_{Q}\in \partial \DD$ be a~fixed point
satisfying $|x_{Q}-y_{Q}| = \dist(x_{Q}, \partial \DD)$. We consider a~ball 
$B(y_{Q}, \frac{\diam Q}{M})$.
By plumpness condition, there exist a~ball $B\subset B(y_{Q}, \frac{\diam 
Q}{M}) 
\cap \DD$ of radius
$\kappa \frac{\diam Q}{M}$. Let $z$ be its center; as $\widetilde{Q}$ we fix any 
of 
the Whitney cubes from $\wh(\DD)$ containing $z$.

Let $z$ be a~point as above. Then
\[
 \kappa\frac{\diam Q}{M} \leq \dist(z, \partial \DD) \leq \frac{\diam Q}{M},
\]
and hence by properties of Whitney cubes
\[
  \frac{\kappa}{5M} \diam Q \leq  \diam \widetilde{Q} \leq  \frac{\diam Q}{M}.
\]
Furthermore, for $x\in Q$ and $w\in \widetilde{Q}$
\begin{align}
|x-w|
 &\leq |x-x_{Q}| + |x_{Q}-y_{Q}| + |y_{Q}-z| + |z-w| \leq (1+5+1+1)\diam Q \\
&\leq 8\dist(Q, \partial \DD). \label{eq:q1Q}
\end{align}
To summarize, the five numbers $\diam(Q)$, $\diam(\widetilde{Q})$, 
$\dist(Q,\partial 
\DD)$,  $\dist(\widetilde{Q},\partial \DD)$, $\dist(Q,\widetilde{Q})$ are 
comparable 
with constants depending only on $\kappa$ and $M$.
\end{proof}

\cite[Remark 3.7]{Triebel2008} provides an example of an interior 
thick set 
$\Omega$ such that $|\partial \Omega|>0$.
It follows from \autoref{l.plump} that such $\Omega$ is not plump. This 
example is however not completely satisfactory in our case, since in our 
results we assume that $|\partial \Omega|=0$. Therefore we provide another 
example.

\begin{example*}
Consider annuli $A_n = \{x\in \R^2: 2^{-n-1} \leq |x| < 2^{-n}\}$ and let 
$a_n=2^{-n-1}/n$, where $n=1,2,\ldots$.
Let $O_n\subset A_n$ be a~maximal set such that balls centered at points from 
$O_n$ with radii $a_n$ are pairwise disjoint and contained in $A_n$. Clearly 
$O_n\neq \emptyset$. Set
\[
 \Omega = \bigcup_{n=1}^\infty \bigcup_{x\in O_n} B(x, \frac{a_n}{2}).
\]
It is easy to observe that $\Omega$ is both 
interior and exterior thick. However, the largest ball that is 
contained in $B(0,2^{-n})$, has a radius smaller than $3 a_n/2$. Since  
$(3 a_n/2)/(2^{-n}) = 3/(4n)\to 0$, the set $\Omega$ 
is not plump. Moreover, $|\partial\Omega|=0$.
\end{example*}

\section{Proof \texorpdfstring{of \autoref{thm:norm}}{}}\label{sec:proof1}
Let $f \in L^p(\R^d)$, $1 \leq p < \infty$. Recall the definition $\delta_z = \operatorname{dist}(z, \partial \Omega)$ 
for $z \in \R^d$. The first inequality in 
\eqref{eq:comp-semi}
follows from the fact that $|x-y| + \delta_x + \delta_y \leq 3|x-y|$ for $x\in 
\Omega$ and $y\in \Omega^c$. This implies
\begin{align*}
3^{-d-p}  \int_{ \Omega} \int_{\Omega^c} 
\frac{|f(x)-f(y)|^p}{|x-y|^{d+sp}}\,dy\,dx \leq 
  \int_{ \Omega \cup \Omext_{\inr(\Omega)} } \int_{\Omega^c} 
\frac{|f(x)-f(y)|^p}{(|x-y|+\delta_x+\delta_y)^{d+sp}}\,dy\,dx \,.
\end{align*}
This estimate implies the desired inequality. 

The remainder of this section is devoted to the proof of the 
second inequality. We observe that
\begin{align*}
 & \int_{ \{x\in \R^d: \delta_x < \inr(\Omega)\} } \int_{\Omega^c} 
\frac{|f(x)-f(y)|^p}{(|x-y|+\delta_x+\delta_y)^{d+sp}}\,dy\,dx\\
&\leq
  |f|_{ \WspOm}^p   
+ \int_{  \{x\in \Omega^c: \delta_x < \inr(\Omega) \} } \int_{\Omega^c} 
\frac{|f(x)-f(y)|^p}{(|x-y|+\delta_x+\delta_y)^{d+sp}}\,dy\,dx
=: |f|_{ \WspOm}^p + I.
\end{align*}
We note that if $x\in \interior \Omega^c$ satisfies $\delta_x<\inr(\Omega)$, 
then a Whitney cube $Q \in \wh(\interior \Omega^c)$
containing $x$ satisfies $\diam Q \leq \dist(Q,\partial \Omega) < 
\inr(\Omega)$. 
Moreover, since $\partial \Omega$ has Lebesgue measure zero, we obtain
\begin{equation}\label{eq:doublesum}
I \leq \sum_{Q_1 \in \wh^b}
\sum_{Q_2 \in \wh(\interior \Omega^c) }
 \int_{Q_1} \int_{Q_2} 
\frac{|f(x)-f(y)|^p}{(|x-y|+\delta_x+\delta_y)^{d+sp}}\,dy\,dx,
\end{equation}
where 
\[
\wh^b = \{ Q_1 \in \wh(\interior \Omega^c): \diam Q_1 < \inr(\Omega) \}.
\]
We take $M=1$ in the \autoref{def:Ithick} so that $\widetilde{Q}$ exists for 
all cubes $Q\in \wh^b$,
and let $C$ be the corresponding constant.
Let  $Q_1 \in \wh^b$ and $Q_2 \in \wh(\interior \Omega^c)$. For $y\in Q_2$ and 
$w\in \widetilde{Q}_1$
\begin{align}
|y-w| &\leq \dist(y,Q_1) + \diam Q_1 + \dist(Q_1,\tilde{Q_1}) + \diam 
\tilde{Q_1} \nonumber\\
&\leq \dist(y,Q_1) + (5C+1)\diam Q_1. \label{eq:q2Q}
\end{align}
Recall that for any cube $Q$, its side length is denoted by $\ell(Q)$ and its center by 
$x_Q$. For $x\in Q_1$ and $y\in Q_2$ we denote
\[
 w=w(x,y) = x_{\widetilde{Q}_1} + \left(\frac{x-x_{Q_1}}{2\ell(Q_1)} + 
\frac{y-x_{Q_2}}{2\ell(Q_2)} \right) \ell(\widetilde{Q}_1)
\]
and observe that $w\in \widetilde{Q}_1$.

We come back to estimating the double integral in \eqref{eq:doublesum}
\begin{align*}
\int_{Q_1} \int_{Q_2} & \frac{|f(x)-f(y)|^p}{ (|x-y| + \delta_x + 
\delta_y)^{d+sp}}\,dy\,dx\\
&\leq
(2^{p-1}\vee 1)
\int_{Q_1} \int_{Q_2} \frac{|f(x)-f(w(x,y))|^p}{ (\dist(Q_1,Q_2) + \dist(Q_1, 
\partial\Omega))^{d+sp}}\,dy\,dx \\
&\quad +
(2^{p-1}\vee 1)
\int_{Q_1} \int_{Q_2} \frac{|f(w(x,y))-f(y))|^p}{ (\dist(y,Q_1) +\dist(Q_1, 
\partial\Omega))^{d+sp}}\,dy\,dx \\
&=:(2^{p-1}\vee 1) (I_1(Q_1, Q_2) +I_2(Q_1, Q_2)).
\end{align*}
To estimate $I_1(Q_1, Q_2)$, we change the variable $y$ to $w=w(x,y)$ in the 
integral and 
obtain
\begin{align*}
I_1(Q_1, Q_2) &\leq \frac{2^d \ell(Q_2)^d}{\ell(\widetilde{Q}_1)^d} \int_{Q_1} 
\int_{\widetilde{Q}_1} 
        \frac{|f(x)-f(w))|^p}{  (\dist(Q_1,Q_2) + \dist(Q_1, 
\partial\Omega))^{d+sp}}\,dw\,dx \\
&\lesssim
    \int_{Q_1} \int_{\widetilde{Q}_1} \frac{|f(x)-f(w))|^p}{ 
|x-w|^{d+sp}}\,dw\,dx 
\cdot 
\frac{\int_{Q_2} \left(1+\frac{\dist(Q_1,Q_2)}{\dist(Q_1, 
\partial\Omega)}\right)^{-d-sp} \,dy}{\ell(\widetilde{Q}_1)^d}
\end{align*}
In the last passage we have used \eqref{eq:q1Q} with $D:= \Omega$, $Q:= Q_1$ and the inequality $s\leq 1$ 
(although any upper bound for $s$ would suffice). We obtain
\begin{align*}
&\sum_{Q_1 \in \wh^b }  \sum_{Q_2\in \wh(\interior \Omega^c)}
 \!\!\!\!\!\! I_1(Q_1, Q_2)\\
&\lesssim
 \sum_{Q_1 \in \wh^b} 
   \int_{Q_1} \int_{\widetilde{Q}_1} \frac{|f(x)-f(w))|^p}{ 
|x-w|^{d+sp}}\,dw\,dx 
\cdot
\sum_{Q_2} \frac{\int_{Q_2} \left(1+\frac{\dist(Q_1,Q_2)}{\dist(Q_1, 
\partial\Omega)}\right)^{-d-sp} \,dy}{\ell(\widetilde{Q}_1)^d}.
\end{align*}
By properties of Whitney cubes,
\begin{align*}
\sum_{Q_2}& \frac{\int_{Q_2} \left(1+\frac{\dist(Q_1,Q_2)}{\dist(Q_1, 
\partial\Omega)}\right)^{-d-sp} \,dy}{\ell(\widetilde{Q}_1)^d}\\
&\leq
\ell(\widetilde{Q}_1)^{-d} c(d) \int_{\R^d} \left(1+ 
\frac{|y-x_{Q_1}|}{\dist(Q_1, 
\partial\Omega)}\right)^{-d-sp}\,dy \\
&=\Big( \frac{\dist(Q_1, \partial\Omega)}{\ell(\widetilde{Q}_1)} \Big)^d c(d)
\int_{\R^d} (1+|z|)^{-d-sp}\,dz
\leq \frac{c(d,C)}{s},
\end{align*}
where the constant $c(d,C)$ depends only on $d$ and $C$, but not on the cube 
$Q_1$. Thus by \autoref{r:overlap}
\begin{equation}\label{eq:sum-norm}
\sum_{Q_1 \in \wh^b }  \sum_{Q_2\in \wh(\interior \Omega^c)}
 \!\!\!\!\!\! I_1(Q_1, Q_2)
\leq \frac{c(d,p,C)}{s} |f|_{ \WspOm}^p.
\end{equation}
We are left with estimating $I_2(Q_1, Q_2)$. We interchange the order of 
integration and 
change the variable $x$ to $w=w(x,y)$.
By \eqref{eq:q2Q}, this gives us
\begin{align*}
I_2(Q_1, Q_2) &\leq \frac{2^d \ell(Q_1)^d}{\ell(\widetilde{Q}_1)^d} \int_{Q_2} 
\int_{\widetilde{Q}_1} 
        \frac{|f(w)-f(y))|^p}{  (\dist(y,Q_1) + \dist(Q_1, 
\partial\Omega))^{d+sp}}\,dw\,dy \\
&\leq c(d,p,C) \int_{Q_2} \int_{\widetilde{Q}_1} 
        \frac{|f(w)-f(y))|^p}{ |w-y|^{d+sp}}\,dw\,dy.
\end{align*}
By \autoref{r:overlap}, we get an estimate of the form \eqref{eq:sum-norm} 
for $I_2(Q_1, Q_2)$ instead of 
$I_1(Q_1, Q_2)$. The proof is complete. 

Note that the constant in \autoref{thm:norm} depends on $\Omega$ only 
through $d$ and the constant $C$ from \autoref{def:Ithick} taken for 
$M=1$.
\qed

\section{Proof \texorpdfstring{ of \autoref{thm:extension}}{}}\label{sec:proof2}
We may assume that $\Omega\neq \emptyset$.
We fix $M=1$ if $\inr(\Omega^c)=\infty$ and $M=\frac{2\sqrt{d} 
  \inr(\Omega)}{\inr(\Omega^c)}$ if $\inr(\Omega^c)<\infty$,
and we fix $\lambda = 1/125$.
We take reflected cubes and the constant $C$ as in \autoref{def:Ithick} and 
\autoref{rem:lambda}
for these particular choices of $M$ and $\lambda$. 
By \autoref{rem:Ithickmod},
the  reflected cubes satisfy $\eqref{e:Ithickmod}$ with $D=\Omega^c$. 

\subsection{Definition of the extension.}
Let $Q_0=[0,1]^d$. We fix a~function $\psi_0\in C_c^\infty(Q_0^*)$ such that 
$\psi_0=1$  on $Q_0$ and $0\leq \psi_0 \leq 1$. We shift and rescale this 
function to other cubes, i.e., we let
\[
\psi_Q(x)= \psi_0\left(\frac{x-x_{Q}}{\ell(Q)} + x_{Q_0}\right), \quad x\in 
\R^d.
\]
Recall from \autoref{subsec:whitney} that for $Q\in\wh(\DD)$ we have $\diam(Q)\le \dist(Q,\partial \DD)\le 4\diam(Q)$. For any cube $Q$, its side length is denoted by $\ell(Q)$ and its center by $x_Q$. By $Q^*$ we denote a cube with the same center as $Q$, but side length 
$\ell(Q^*)=(1+\varepsilon)\ell(Q)$, where $0<\varepsilon<1/4$ is fixed as above.

We consider the following family of functions
\[
 \phi_Q(x) =\frac{\psi_Q(x)}{\sum_{R\in \wh(\Omega)} \psi_R(x)},\quad x\in \R^d.
\]
Thus $\phi_Q\geq 0$ and $\sum_{Q\in \wh(\Omega)} \phi_Q = \ind_\Omega$. Let $f 
\in L^1_{loc}(\Omega^c)$. Let 
\[
a_Q = \frac{1}{|\widetilde{Q}|} \int_{\widetilde{Q}} f(x)\,dx, \quad Q\in 
\wh(\Omega)\,.
\]
Note that the reflected cube $\widetilde{Q}$ is well defined thanks to the 
choice 
of 
$M$.
We extend a given function $f \in L^1_{loc}(\Omega^c)$ from $\Omega^c$ to 
$\R^d$ 
by defining $\ext(f)$ as follows:
\begin{align*}
 \ext(f)(x) = 
 \begin{cases}
 \sum_{Q\in \wh(\Omega)} a_Q \phi_Q(x) \quad &\text{ if } x \in \Omega \,, \\
 f(x) &\text{ if } x \in \Omega^c \,.
 \end{cases}
\end{align*}
Let $\neio(Q) = \{R\in \wh(\Omega): R\cap Q^* \neq\emptyset \}$ be the 
collection 
 of Whitney cubes intersecting $Q$. Observe that for $x\in Q_1\in \wh(\Omega)$ 
and any $t \in \R$
\begin{equation}\label{e:fdecomp}
 \ext(f)(x) = \sum_{Q\in \neio(Q_1)} a_Q \phi_Q(x) = t + \sum_{Q\in \neio(Q_1)} 
(a_Q-t) \phi_Q(x).
\end{equation}

\subsection{A remark on reflected cubes}\label{subs:reflected}
Let
\[
\wh^{<\delta}(\Omega) = \{ Q:\wh(\Omega) : \dist(Q,\partial \Omega) < \delta\}.
\]
Let us note that if $Q_1 \in \wh^{<\delta}(\Omega)$, $Q_2\in \neio(Q_1)$ and 
$Q_3\in\neio(Q_2)$,
then
\[
\diam Q_3 \leq 5\diam Q_2 \leq 25\diam Q_1.
\]
Therefore, if $z\in \widetilde{Q}_3$, then
\[
\dist(z,\partial \Omega) \leq 5\diam \widetilde{Q}_3 \leq 5\lambda \diam Q_3 
\leq
125\lambda \diam Q_1 \leq 125\lambda \dist(Q_1, \partial\Omega) < \delta,
\]
that is, $\widetilde{Q}_3\subset \Omext_\delta$. In particular, 
$\widetilde{Q}_1, \widetilde{Q}_2 \subset \Omext_\delta$, because we may take 
$Q_3=Q_2$ or $Q_3=Q_1$.

\subsection{\texorpdfstring{An estimate of $|a_{Q_1}-a_{Q}|^p$}{First 
estimate}.}
We claim that, for $Q_1,Q \in \wh(\Omega)$,
\begin{align}\label{e:aQ-est}
|a_{Q_1}-a_{Q}|^p \lesssim \frac{(\dist(Q_1,Q)+\ell(Q_1)+\ell(Q))^{d+sp}}{|Q_1| 
|Q|}
\,|f|_{\widetilde{Q}_1, \widetilde{Q}}^{s,p}\,.
\end{align}
Indeed,
\begin{align*}
|a_{Q_1}-a_{Q}| &=
\frac{1}{|\widetilde{Q}_1| |\widetilde{Q}|}
\left|
 |\widetilde{Q}| \int_{\widetilde{Q}_1} f(y) \,dy -
 |\widetilde{Q}_1| \int_{\widetilde{Q}} f(x) \,dx \right| \\
&\leq
\frac{1}{|\widetilde{Q}_1| |\widetilde{Q}|}
  \int_{\widetilde{Q}_1} \int_{\widetilde{Q}} |f(y)-f(x)| \,dy \, dx \,.
\end{align*}
From $|Q_j| \lesssim |\widetilde{Q}_j|$, and Jensen inequality we deduce
\begin{align*}
|a_{Q_1}-a_{Q}|^p &\lesssim
\frac{1}{|Q_1||Q|}
  \int_{\widetilde{Q}_1} \int_{\widetilde{Q}} |f(y)-f(x)|^p \,dy \, dx,
\end{align*}
and the claim follows.

\subsection{\texorpdfstring{An estimate of $|\phi_Q|_{Q_1,Q_2}^{s,p}$ for 
$s\leq 1$ and arbitrary cubes $Q, Q_1$, $Q_2$.}{Second estimate}}
It is easy to check that $|\nabla \phi_Q| \lesssim \ell(Q)^{-1}$. Therefore
$|\phi_Q(x) - \phi_Q(y)| \lesssim \ell(Q)^{-1}|x-y| \wedge 1$ for all $x,y$. 
As a result, we obtain
\begin{align}
 |\phi_Q|_{Q_1,Q_2}^{s,p} &\lesssim  \int_{Q_1} \int_{Q_2} 
\frac{\ell(Q)^{-p}|x-y|^p \wedge 1}{|x-y|^{d+sp}}\,dy\,dx \nonumber\\ 
 & \lesssim  |Q_1| |Q_2| \Big(\frac{\ell(Q_1)^{p-sp-d} \ell(Q)^{-p}}{1-s} 
\wedge 
\frac{\ell(Q_2)^{p-sp-d} \ell(Q)^{-p}}{1-s} \wedge \dist(Q_1,Q_2)^{-d-sp} 
\Big)\,. 
\label{e:phiQ-est}
\end{align}

We note that the above inequality for $s=1$ is nontrivial only if 
$\dist(Q_1,Q_2)>0$.

\subsection{Proof  of part (b)\texorpdfstring{, formula \eqref{e:int-int}}{}}
It holds
\begin{align}\label{e:f-omega-omega}
  |\ext(f)|_{\Omint_\delta, \Omint_\delta}^{s,p} &\leq
 \sum_{Q_1\in  \wh^{<\delta}(\Omega)} \sum_{Q_2\in  \wh^{<\delta}(\Omega)}  
|\ext(f)|_{Q_1,Q_2}^{s,p}.
\end{align}
For $Q_1$, $Q_2\in \wh^{<\delta}(\Omega)$, we use \eqref{e:fdecomp} twice with 
$t=a_{Q_1}$ and obtain
\begin{align*}
|\ext(f)|_{Q_1,Q_2}^{s,p} 
 &\lesssim 
 \sum_{Q\in \neio(Q_1)\cup\neio(Q_2) } |a_Q-a_{Q_1}|^p |\phi_Q|_{Q_1,Q_2}^{s,p} 
\, .
\end{align*}
If additionally $Q_2 \in \neio(Q_1)$, then for $Q\in \neio(Q_1)\cup\neio(Q_2)$
\begin{align*}
 |a_Q-a_{Q_1}|^p |\phi_Q|_{Q_1,Q_2}^{s,p} &\lesssim  
\frac{(\dist(Q_1,Q)+\ell(Q_1)+\ell(Q))^{d+sp}}{|Q_1| |Q|}
\,|f|_{\widetilde{Q}_1, \widetilde{Q}}^{s,p} 
 \cdot \frac{|Q_1| |Q_2| \ell(Q_1)^{-sp-d}}{1-s} \\
&\lesssim \frac{\,|f|_{\widetilde{Q}_1, \widetilde{Q}}^{s,p}}{1-s}\,. 
\end{align*}
Otherwise, if  $Q_2\in \wh^{<\delta}(\Omega) \setminus \neio(Q_1)$, then 
$\dist(Q_1,Q_2)>0$ and consequently $\ell(Q_1)$, $\ell(Q_2) \lesssim 
\dist(Q_1,Q_2)$. Then
for  $Q\in \neio(Q_1)$
\begin{align*}
 |a_Q-a_{Q_1}|^p |\phi_Q|_{Q_1,Q_2}^{s,p} &\lesssim
 \frac{(\ell(Q_1)+\ell(Q))^{d+sp}}{|Q_1| |Q|}
\,|f|_{\widetilde{Q}_1, \widetilde{Q}}^{s,p}  \cdot \frac{|Q_1| 
|Q_2|}{\dist(Q_1,Q_2)^{d+sp}}\\
&\lesssim  \frac{\ell(Q_1)^{sp} |Q_2|\, |f|_{\widetilde{Q}_1, 
\widetilde{Q}}^{s,p}}{\dist(Q_1,Q_2)^{d+sp}}.
\end{align*}
On the other hand, for $Q\in \neio(Q_2)$
\begin{align*}
 |a_Q-a_{Q_1}|^p |\phi_Q|_{Q_1,Q_2}^{s,p} &\lesssim
 \frac{(\dist(Q_1,Q)+\ell(Q_1)+\ell(Q))^{d+sp}}{|Q_1| |Q|}
\,|f|_{\widetilde{Q}_1, \widetilde{Q}}^{s,p}  \cdot \frac{|Q_1| 
|Q_2|}{\dist(Q_1,Q_2)^{d+sp}}\\
&\lesssim |f|_{\widetilde{Q}_1, \widetilde{Q}}^{s,p}.
\end{align*}
Let $\wh^{<\delta}_{\mathcal{N}}(\Omega) = \{ Q\in \wh(\Omega) : \neio(Q) \cap 
\wh^{<\delta}(\Omega) \neq\emptyset \}$.
Combining the above inequalities yields
\begin{align}
  |\ext(f)|_{\Omint_\delta, \Omint_\delta}^{s,p} 
&\lesssim
 \sum_{Q_1\in  \wh^{<\delta}(\Omega)} \bigg( \sum_{Q_2\in  \neio(Q_1)} 
\sum_{Q\in 
\neio(Q_1)\cup\neio(Q_2) } \frac{\,|f|_{\widetilde{Q}_1, 
\widetilde{Q}}^{s,p}}{1-s} 
\nonumber\\
&\qquad\qquad\qquad+ \sum_{Q_2\in  \wh^{<\delta}(\Omega) \setminus \neio(Q_1)}  
\sum_{Q\in 
\neio(Q_1)} \frac{\ell(Q_1)^{sp} |Q_2|\, |f|_{\widetilde{Q}_1, 
\widetilde{Q}}^{s,p}}{\dist(Q_1,Q_2)^{d+sp}}\nonumber\\
&\qquad\qquad\qquad+ \sum_{Q_2\in  \wh^{<\delta}(\Omega) \setminus \neio(Q_1)}  
\sum_{Q\in 
\neio(Q_2)} |f|_{\widetilde{Q}_1, \widetilde{Q}}^{s,p} \bigg) \nonumber\\
&\lesssim
 \sum_{Q_1\in  \wh^{<\delta}(\Omega)} \bigg( \sum_{Q: \neio(Q)\cap\neio(Q_1) 
\neq\emptyset 
} \frac{\,|f|_{\widetilde{Q}_1, \widetilde{Q}}^{s,p}}{1-s} \nonumber\\
&\qquad\qquad\qquad+   \sum_{Q\in \neio(Q_1)} 
 \, |f|_{\widetilde{Q}_1, \widetilde{Q}}^{s,p} \int_{\Omega\setminus \bigcup 
\neio(Q_1)} 
\frac{\ell(Q_1)^{sp} \,dx}{\dist(x,Q_1)^{d+sp}}\nonumber\\
&\qquad\qquad\qquad+ \sum_{Q_2\in  \wh^{<\delta}_{\mathcal{N}}(\Omega) \setminus 
\{Q_1\}} 
|f|_{\widetilde{Q}_1, \widetilde{Q}_2}^{s,p} \bigg) \nonumber\\
&\lesssim
 \frac{  |f|_{\Omext_\delta, \Omext_\delta}^{s,p}}{s(1-s)}. \label{e:f-o-o}
\end{align}
The fact that in the last expression there are sets $\Omext_\delta$ follows from 
a~remark
in~\autoref{subs:reflected}.

\subsection{Proof of part (b)\texorpdfstring{, formula \eqref{e:int-ext}}{}}
It holds
\begin{align}\label{e:aim}
  |\ext(f)|_{\Omint_\delta, \Omext_\varepsilon}^{s,p} &\leq
 \sum_{Q_1\in  \wh^{<\delta}(\Omega)} \sum_{Q_2\in  \wh^{<\varepsilon}(\interior 
\Omega^c)}  
|\ext(f)|_{Q_1,Q_2 \cap \Omext_\varepsilon}^{s,p}.
\end{align}
Now let  $Q_1\in \wh^{<\delta}(\Omega)$ and $Q_2\in 
\wh^{<\varepsilon}(\interior \Omega^c)$. We again 
use 
\eqref{e:fdecomp}  with $t=a_{Q_1}$,
\begin{align*}
|\ext(f)|_{Q_1,Q_2  \cap \Omext_\varepsilon}^{s,p} 
 & =
\int_{Q_1} \int_{Q_2  \cap \Omext_\varepsilon} \frac{\big| \sum_{Q\in 
\neio(Q_1)} (a_Q - 
a_{Q_1})\phi_Q(x) + a_{Q_1}- f(y)\big|^p}{(|x-y| +\delta_x+\delta_y)^{d+sp}} 
\,dy\,dx\\
 &\lesssim
 \sum_{Q\in \neio(Q_1)} \int_{Q_1} \int_{Q_2}  \frac{\big| (a_Q - 
a_{Q_1})(\phi_Q(x)-\phi_Q(y)) \big|^p}{(|x-y| +\delta_x+\delta_y)^{d+sp}} 
\,dy\,dx \\
 &\qquad+  \dist(Q_1,Q_2)^{-d-sp} \int_{Q_1} \int_{Q_2 \cap \Omext_\varepsilon} 
\Big|  a_{Q_1}- 
f(y)\Big|^p\,dy\,dx =: A+ \frac{B}{\dist(Q_1,Q_2)^{d+sp}}.
\end{align*}
The first term above is estimated  using \eqref{e:aQ-est} and 
\eqref{e:phiQ-est},
\begin{align*}
 A &\lesssim \sum_{Q\in \neio(Q_1)}
\frac{(\dist(Q_1,Q)+\ell(Q_1)+\ell(Q))^{d+sp}}{|Q_1| |Q|}
\,|f|_{\widetilde{Q}_1, \widetilde{Q}}^{s,p}\, 
\frac{|Q_1||Q_2|}{\dist(Q_1,Q_2)^{d+sp}} 
\\
&\lesssim \sum_{Q\in \neio(Q_1)}
|f|_{\widetilde{Q}_1, \widetilde{Q}}^{s,p}\, \frac{ \ell(Q_1)^{sp} 
|Q_2|}{\dist(Q_1,Q_2)^{d+sp}}.
\end{align*}
For the second term,
\begin{align*}
B &= |Q_1| \int_{Q_2 \cap \Omext_\varepsilon} \left| \frac{1}{|\widetilde{Q}_1|} 
\int_{\widetilde{Q}_1} 
(f(x)-f(y)) \,dx \right|^p \,dy \lesssim \int_{Q_2 \cap \Omext_\varepsilon}  
\int_{\widetilde{Q}_1} 
|f(x)-f(y)|^p \,dx  \,dy \\
&\lesssim \dist(Q_1,Q_2)^{d+sp} 
   \int_{Q_2 \cap \Omext_\varepsilon}  \int_{\widetilde{Q}_1} 
\frac{|f(x)-f(y)|^p}{(|x-y| 
+\delta_x+\delta_y)^{d+sp}} \,dx  \,dy = \dist(Q_1,Q_2)^{d+sp} |f|_{Q_2 \cap 
\Omext_\varepsilon, 
\widetilde{Q}_1}^{s,p}.
\end{align*}
Inequalities obtained for $A$ and $B$ together with \eqref{e:aim} yield
\begin{align}
 |\ext(f)|_{\Omint_\delta, \Omext_\varepsilon}^{s,p} 
&\lesssim
  \sum_{Q_1\in  \wh^{<\delta}(\Omega)} 
\bigg(
\sum_{Q\in \neio(Q_1)} 
 |f|_{\widetilde{Q}_1, \widetilde{Q}}^{s,p} \int_{\Omega^c}\frac{\ell(Q_1)^{sp} 
\,dx}{\dist(x,Q_1)^{d+sp}} +
\sum_{Q_2\in  \wh^{<\varepsilon}(\interior \Omega^c)}
 |f|_{Q_2\cap \Omext_\varepsilon, \widetilde{Q}_1}^{s,p} \bigg) \nonumber\\
&\lesssim \frac{1}{s} |f|_{\Omext_\delta, \Omext_\varepsilon}^{s,p}, 
\label{e:f-o-oc}
\end{align}
since by~\autoref{subs:reflected},
the cubes $\widetilde{Q}_1$ and $\widetilde{Q}$ above are contained in 
$\Omext_\delta \subset \Omext_\varepsilon$.

\subsection{Proof of part (b)\texorpdfstring{, formulas \eqref{e:simpl1} and 
\eqref{e:simpl2}}{}}
Formula \eqref{e:simpl1} follows directly by taking $\delta=\inr \Omega$ and 
$\varepsilon=\infty$ in \eqref{e:int-ext} and enlarging the right hand side;
alternatively, one may also apply   \eqref{e:int-ext} to 
$\delta=\varepsilon=\infty$.

To prove \eqref{e:simpl2} we proceed as follows,
\begin{align*}
  |\ext(f)|_{\R^d, \R^d}^{s,p} &=
  |\ext(f)|_{\Omega, \Omega}^{s,p} +  2|\ext(f)|_{\Omega, \Omega^c}^{s,p} + 
|f|_{\Omega^c, \Omega^c}^{s,p} 
  \leq \frac{c}{s(1-s)} |f|_{\Omega^c, \Omega^c}^{s,p},
\end{align*}
by \eqref{e:int-ext} and \eqref{e:int-int} applied to 
$\delta=\varepsilon=\infty$.

We note that the constants depend on $\Omega$ only through $d$, $M$ and $C$.
\subsection{Proof of part (a).}
The smoothness of $\ext(f)$ on $\Omega$ follows directly from the definition.
The proof of the second part is  omitted as it is  straightforward, it is based
on the fact that if the cubes $Q\in \wh(\Omega)$ approach $z\in \partial 
\Omega$,
then so do the reflected cubes $\widetilde{Q}$.
\qed

\subsection{Proof of part (c).}
If $p=\infty$, then
\begin{align*}
  \|\ext(f)\|_{L^\infty(\Omega)}
  &=
  \sup_{Q_1 \in \wh(\Omega)} \, \sup_{x\in Q_1} \left| \sum_{Q\in \neio(Q_1)} 
a_Q \phi_Q(x) \right| \\
  &\leq \sup_{Q_1 \in \wh(\Omega)} \# \neio(Q_1) \cdot 
\|f\|_{L^\infty(\Omext_{\inr(\Omega)})}
  \lesssim \|f\|_{L^\infty(\Omega^c)}.
\end{align*}

Now let $p<\infty$.
We first observe that
\begin{equation}\label{eq:weightcomparison}
 |x|+1 \asymp |\tilde{x}|+1 \qquad \textrm{whenever  $x\in Q\in \wh(\Omega)$ 
and 
$\tilde{x}\in \widetilde{Q} \in \wh(\interior \Omega^c)$,}
\end{equation}
with constants dependent only on the domain $\Omega$. Indeed, let 
$R=\dist(0,\partial \Omega)$, then
\[
|\tilde{x}| \leq |\tilde{x}-x| + |x| \lesssim \diam Q + |x|,
\]
and
\[
\diam Q \leq \dist(Q,\partial \Omega) \leq |x-0| + \dist(0,\partial \Omega)  = 
|x|+R,
\]
so $|\tilde{x}|+1 \lesssim |x|+1$, as claimed. The proof of the opposite 
estimate is similar and omitted.

Let $\omega(x)=(1+|x|)^\beta$. Since the numbers $\# \neio(Q_1)$ for $Q_1 \in 
\wh(\Omega)$ are bounded from above by a~constant depending only on the domain 
$\Omega$, we obtain
\begin{align}
  \|f\|_{L^p(\Omega,\,\omega(x)dx)}^p
  &= \sum_{Q_1 \in \wh(\Omega)} \, \int_{Q_1} \left| \sum_{Q\in \neio(Q_1)} a_Q 
\phi_Q(x) \right|^p \omega(x)\,dx\nonumber\\
  &\lesssim
  \sum_{Q_1 \in \wh(\Omega)} \sum_{Q\in \neio(Q_1)} |a_Q|^p  \int_{Q_1}  
|\phi_Q(x)|^p \omega(x)\,dx \label{eq:pestimate}
  \end{align}
By Jensen inequality, comparability of the sizes of cubes $Q$, $\widetilde{Q}$ 
and 
$Q_1$ as in the sum above, and \eqref{eq:weightcomparison} we can estimate each 
summand as follows
\begin{align*}
  |a_Q|^p  \int_{Q_1}  |\phi_Q(x)|^p \omega(x)\,dx
  &\leq \frac{1}{|\widetilde{Q}|} \int_{\widetilde{Q}} |f(x)|^p\,dx \cdot |Q_1| 
\cdot 
\sup_Q \omega
  \lesssim  \int_{\widetilde{Q}} |f(x)|^p \omega(x)\,dx.
\end{align*}
Using boundedness of $\# \neio(Q_1)$, and \autoref{r:overlap} we obtain from 
the estimate \eqref{eq:pestimate} the following estimate:
\begin{align*}
  \|f\|_{L^p(\Omega,\,\omega(x)dx)}^p
&\lesssim
  \sum_{Q_1 \in \wh(\Omega)} \sum_{Q\in \neio(Q_1)} \int_{\widetilde{Q}} 
|f(x)|^p 
\omega(x)\,dx
\lesssim
  \sum_{Q \in \wh(\Omega)} \int_{\widetilde{Q}} |f(x)|^p \omega(x)\,dx \\
&\lesssim  
  \|f\|_{L^p(\Omext_{\inr(\Omega)},\,\omega(x)dx)}^p.
\end{align*}
This completes the proof of part (c) and thus the proof of 
\autoref{thm:extension}.
\qed


\def\cprime{$'$}

\end{document}